\documentclass{article}

\usepackage[utf8x]{inputenc}
\usepackage{amsmath}
\usepackage{amssymb}
\usepackage{amsthm}
\usepackage{color}
\usepackage{algorithm}
\usepackage[noend]{algpseudocode}
\usepackage{pgf,tikz}
\usetikzlibrary{arrows,shapes,positioning}
\usepackage{graphicx}

\usepackage{geometry}
\newtheorem{thm}{Theorem}

\begin{document}

\newcommand\inner[2]{\langle#1 ,#2 \rangle}
\newcommand\norm[1]{\mid \mid #1 \mid \mid}

\title{Representation and Characterization of Non-Stationary Processes by Dilation Operators and Induced Shape Space Manifolds}

\author{Maël Dugast, Guillaume Bouleux, Eric Marcon}
\date{}

\maketitle

%     =================================================       INTRO        ============================================================

\begin{abstract}
We have introduce a new vision of stochastic processes through the geometry induced by the dilation. The dilation matrices of a given processes which are obtained by a composition of rotations matrices, contain the measure information in a condensed way. Particularly interesting is the fact that the obtention of dilation matrices is regardless of the stationarity of the underlying process. When the process is stationary, it coincide with the Naimark Dilation. When the process is non-stationary, it returns a set of rotation matrices. In particular, the periodicity of the correlation function that may appear in some classes of signal is transmitted to the set of dilation matrices. These rotation matrices, which can be arbitrarily close to each other depending on the sampling or the rescaling of the signal are seen as a distinctive feature of the signal.
In order to study this sequence of matrices, and guided by the possibility to rescale the signal, the correct geometrical framework to use with the dilation's theoretic results is the space of curves on manifolds, that is the set of all curve that lies on a based manifold. To give a complete sight about the space of curve, a metric and the derived geodesic equation are provided. The general results are adapted to the more specific case where the base manifold is the Mie group of rotation matrices.
The notion of the shape of a curve can be formalized as the set of equivalence classes of curves given by the quotient space of the space of curves and the increasing diffeomorphisms. The metric in the space of curve naturally extent to the space of shapes and enable comparison between shapes.
\end{abstract}

\section{Introduction}

The Analysis and/or the representation of non-stationary processes has been tackled during 4 or 5 decades for now  by time-scale/time frequency analysis \cite{Flandrin1999, Auger2013}, by Fourier-like representation when the processes belong to the periodically-correlated subclass \cite{Napolitano2016,hurd2007} or by partial correlation coefficients (pacors) series \cite{lambert_lacroix_extension_2005,degerine_characterization_2003} to cite a few. One of the advantages of dealing with parcors resides in their strong  relation to the measure of the process by the one-to-one relation with correlation coefficients \cite{Desbouvries2003,yang_riemannian_2010}. They consequently appear explicitly in the Orthogonal Polynomial on the Real Line (OPRL)/Unit Circle (OPUC) decomposition of the measure \cite{Bingham2011,Simonbooks} and are the elements for the construction of dilation matrices that appear in the CMV/GGT \cite{simon_cmv_2007}, for the Schur flows problem with upper Hessenberg matrices \cite{ammar_constructing_1991} which are also seen in the literature as evolution operators \cite{Simonbooks} or shift operator \cite{masani_dilations_1978}, and appear finally in the state-space representation \cite{constantinescu_schur_1995,Desbouvries1996}. The dilation theory takes its roots in the Operators theory \cite{sz.-nagy_harmonic_2010} which bridges the process's measure and unitary operators. In its simplest version, the dilation theory corresponds to the Naimark dilation \cite{sz.-nagy_harmonic_2010,arsene_structure_1985}, and state that given a sequence of correlation coefficients, there exists a unitary matrix $W$ such that $ R_{n} \triangleq (1\ 0 \ 0  \cdots ) W^{n} (1\ 0 \ 0  \cdots )^{T} $ where $ \cdot^{T} $ denotes the transposition. When the process is not stationary, its associated correlation matrix is no more Toeplitz structured, a set of matrices is required \cite{constantinescu_schur_1995} and the previous expression becomes $ R_{i,j} \triangleq (1\ 0 \ 0  \cdots ) W_{i+1}W_{i+2} \cdots W_{j} (1\ 0 \ 0  \cdots )^{T} $. The matrices $W_i$ are theoretically understood as infinite rotation matrices, which become finite when the correlation coefficients sequence is itself finite. In that particular case, the matrices $W_i$ belong to $SO(n)$ or $SU(n)$, the special orthogonal or unitary group respectively and the process's measure is totally described by the set of $W_i$. As a consequence, the measure of the process is beautifully characterized for the non-stationary case, by a sampled trajectory induced by the dilation matrices on the appropriate Lie Group. When the process is periodically correlated, the sequence of parcors inherits of the periodicity and the sequence of dilation matrices becomes periodic as well, we obtain consequently a closed path as illustrated in Fig.\ref{Fig1}. Characterizing the time varying measure of the process is now tackled by studying curves (or sampled curves) on special groups.\\
Information geometry is now a fundamental approach for describing stochastic processes \cite{Manton2013}. The second order statistical properties/moments may be analyzed, characterized and compared \cite{Barbaresco2013,Balaji2014} to improve estimation \cite{Sun2014,LipengNing2012} or classification of different processes \cite{Kalunga2015}. When dealing with density estimation \cite{Jiang2012}, the space of $n \times n$ symmetric matrices $\mathcal{S}ym(n)$ is generally preferred and many developments  have been proposed under the Semi Positive Definite (SPD) assumption \cite{pennec_intrinsic_2006,pennec_riemannian_2006,Smith2005,Chevallier2014} for which the set of SPD matrices constitute a convex half-cone in the vector space of matrices. This leads to give more insights on the Fisher Information metric \cite{Chevallier2014,Jiang2012} or Wasserstein metric \cite{Kolouri2017} and also cope with optimal mass transportation problems \cite{Barbaresco2011}. Many efforts have also been made the last decade to exploit the hyperbolic geometry structure not of the correlation matrices directly but of the related parcors when obtained in stationary conditions \cite{desbouvries_non_euclidean_2003,desbouvries_geometrical_nodate,yang_riemannian_2010,barbaresco_interactions_2009,arnaudon_riemannian_2013}. As the Kullback–Leibler divergence let do, the comparison of stationary processes is then made by comparing curves, whose sampled points are parcors sequences, defined on several copies of the Poincaré disk through geodesics deformation. Treating the non-stationary case has not been tackled to our knowledge with the previous mentioned approaches. In this paper, we hope to initiate interest in filling this gap by extending the representation and the characterization of processes' measure in non-stationary context. First, using the dilation theory approach to give sampled points and next giving the prescribed geodesics equations used for curves or paths comparisons in Lie group.\\ 
To support the reader, some insights on dilation theory are given in section~\ref{section:1}. Practical implementations of dilation matrices according to the operator theory approach \cite{arsene_structure_1985, constantinescu_schur_1995} or  the lattice filter structure approach \cite{kailath_naimark_1986,sayed_recursive_2002} are also discussed and the strong connection between parcors and the dilation matrices are emphasized.
The section~\ref{section:2} focuses on the geometry of the curves induced by the dilation on particular manifolds. The general framework is first introduced by recalling concepts of distances and shape of curves when the ambiant space is not flat. Next, the Square Root Velocity (SRV) functions are developed and adapted to Lie Group and a procedure to compare non-stationary processes through their time evolution trajectory is presented. Finally, a conclusion is drawn in section~\ref{section:conclu} and the reader will find some technical tools in the Appendix section. 

\begin{figure}
\begin{center}
\includegraphics[scale=0.5]{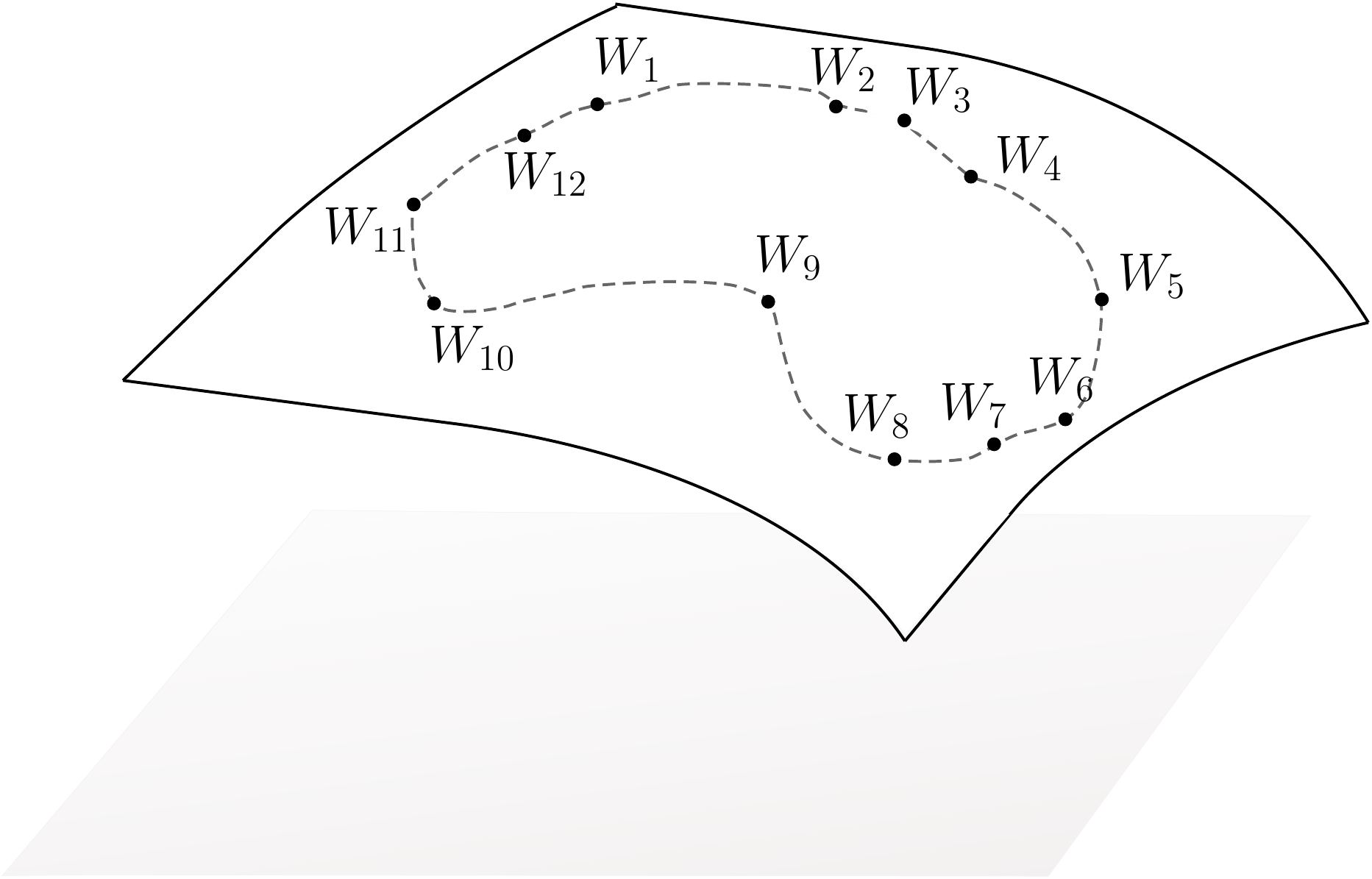}
\end{center}
\caption{Illustration of a sampled closed trajectory drawn in $SO(n)$ or $SU(n)$ which materializes the time varying  of the periodically-correlated measure for a stochastic process. Each $W_i$ are dilation matrix built through the parcors.}\label{Fig1}
\end{figure}
%================================    SECTION  ==========================================

\section{The structure of semi-positive definite matrices and the dilation theory}\label{section:1}

\subsection{ The theory of dilation and the interaction with }

Let us give some insights on the dilation theory. In its fundamental definition, the dilation theory consists for an Hilbert space $\mathcal{H}$ and an operator-valued function $f$, \emph{i.e.} a $\mathcal{L}(\mathcal{H})$-valued function, to find a larger Hilbert space $H$ and an other application $\mathcal{F}$ such that $f$ is the orthogonal projection of $\mathcal{F}$ :
\begin{equation}
f(t) = P_{\mathcal{H}}\mathcal{F}(t), \qquad t \in \mathbb{Z}
\end{equation}

where $P_{\mathcal{H}}$ denotes the orthogonal projection onto the Hilbert space $\mathcal{H}$. The ideas of the dilation theory are : 

\begin{itemize}
\item
there exists a larger space from which the original function (or matrix) is deduced
\item
we can choose the "dilated" function to be simpler. For instance, when dealing with matrices, each of its coefficients can be expressed as the projection of a larger unitary matrix. In this case, we obtain a unitary dilation. This approach has been for example developed in \cite{Makagon1989,Miamee1989,miamee90} for the stationary dilation of periodically-correlated processes.
\end{itemize}

%+++++++++++++++++++++++++++++++++++++++++++++++++++++++
\subsubsection{Dilation and rotation of contractions}

For an operator $T$ on a Hilbert space $ \mathcal{H} $, we denote by $ T^{\ast} $ the adjoint operator, \emph{i.e.} the operator on $ \mathcal{H} $ such that $ \inner{ Tx }{y} = \inner{x}{T^{\ast}y} $ for all $ x,y \in \mathcal{H} $. An operator $T \in \mathcal{L}(\mathcal{H})$ is said to be a contraction if $\norm{T} \leq 1$ where $\norm{\cdot}$ is the operator norm. We deduce the expression for the defect operator $D_{T} = \left(  I - T^{\ast}T \right)^{1/2}$ and its adjoint $D_{T^{\ast}} = \left(  I - T T^{\ast} \right)^{1/2}$. \\
One of the easiest result is that, given a contraction $\Gamma$, the aforementioned unitary Julia operator 
\begin{equation}
J(\Gamma) = 
\begin{pmatrix}
\Gamma & D_{\Gamma^{\ast}} \\
D_{\Gamma} & - \Gamma^{\ast}
\end{pmatrix}
\end{equation}

satisfies, for all $n\in \mathbb{N}$ 
\begin{equation}
\Gamma^{n} = 
\begin{pmatrix}
1 & 0
\end{pmatrix}
J(\Gamma)^{n}
\begin{pmatrix}
1\\
0
\end{pmatrix}.
\end{equation}

In other words, the elementary rotation of a contraction, called consequently the Julia operator, also corresponds to the unitary dilation operator of the contraction. Note that the Julia operator is sometimes called the Halmos extension \cite{masani_dilations_1978} of a contraction. \\
\subsubsection{Dilation and isometries}

Following the idea and the formulation of Naimark, the dilation theory can be restated in terms of dilation of a sequence of operators, or sequence of numbers when the dimension of the underlying Hilbert space is 1. Recall that a sequence of operators $\left\{  R_{n} \right\}_{n=1}^{\infty}$ acting on $\mathcal{H}$ is said to be positive if 
\begin{equation}
\sum_{i,j=0}^{+ \infty}{  \inner{ R_{i-j}h_{i}}{ h_{j} }} \geq 0 \quad for\ all \ h_{i} \in \mathcal{H}_{i}.
\end{equation}
Assuming now that $R_{n}^{\ast} = R_{-n}$ and $R_{0} = I$, leads to the following Toeplitz matrix:
\begin{equation}\label{Positive-Definite matrix}
R^{(m)} =
\begin{pmatrix}
I & R_{1} & \cdots & R_{m-1} \\
R_{1}^{\ast} & I & \cdots & R_{m-2} \\
\cdot & \cdot & \cdots & \cdot \\
\cdot & \cdot & \cdots & \cdot \\
R_{m-1}^{\ast} & R_{m-2}^{\ast} & \cdots & I
\end{pmatrix}
\end{equation}
which is positive-definite. Remark that this matrix can be seen as the correlation matrix of a stationary process, as it is positive and Toeplitz \cite{Bochner1932,arsene_structure_1985,tseng_dilation_2006}. Owing this property, we obtain the following relation 
\begin{equation}\label{Minimal Unitary Dilation}
R_{n} = P_{\mathcal{H}} U^{n}\mid_{\mathcal{H}} , \quad for \ all\ n \geq 0 \; \text{ and } \; U \text{ an isometry on }\mathcal{K}
\end{equation}
due to Naimark dilation Theorem. Furthermore, if $\mathcal{K} = \underset{n \geq 0} {\bigvee} U^{n}\mathcal{H}$ then $U$ is unique up to an isomorphism.\\

\subsubsection{Dilation and measure}

From Bochner's theorem we known that a matrix of type (\ref{Positive-Definite matrix}) can be seen as the Fourier coefficient of a given positive Borelian measure. This is also known as the moment or trigonometric problem \cite{constantinescu_schur_1995}. Therefore, we can restate the dilation problem in term of measure.  If we denote by $ E_{\lambda} $ an operator-valued distribution function on $ [0, 2 \pi[ $ then the function

\begin{equation}
R_{n} = \int_{0}^{2\pi}{ \mathrm{e}^{i n \lambda} dE_{\lambda} }.
\end{equation}

This function is positive-definite and shows the strong correspondance between the spectral measure and the dilation theory. There exists so a unitary operator on a Hilbert space $ \mathcal{K} $ such that $ R_{n} = P_{\mathcal{H}} U(n)$ where $ P_{\mathcal{H}}$ stands for the orthogonal projection. With the spectral representation of unitary operators, $ U = \int_{0}^{2 \pi}{ \mathrm{e}^{i \lambda}d E_{\lambda} } $ and we have
 \begin{equation}
\int_{0}^{2 \pi}{ \mathrm{e}^{i n \lambda} d\inner{E_{\lambda} u}{v} } = \int_{0}^{2 \pi}{ \mathrm{e}^{i n \lambda} d\inner{F_{\lambda} u}{v} }
\end{equation}
or, in an equivalent form :
\begin{equation}
E_{\lambda} = P_{\mathcal{H}} F_{\lambda}.
\end{equation}
Note that the operator-valued measure $ F_{\lambda} $ is in fact an orthogonal projection-valued measure since all its increments are orthogonal.\\ 
Dilations matrices having been introduced, we give now in the next section  a methodology to understand how they are obtained.

%+++++++++++++++++++++++++++++++++++++++++++++++++++++++++++
\subsection{Construction of Dilation Matrices}

As mentioned previously, given a SPD matrix $ R = \left( R_{i,j} \right)_{i,j \in \mathbb{N}} $, it is possible to build a sequence of matrices $ \left\{ W_{i} \right\}_{i \in \mathbb{N}} $ such that 
$ R_{i,j} = \begin{pmatrix}
1 & 0 & 0 & \cdots & 0
\end{pmatrix}W_{i}W_{i+1} \cdots W_{j-1}
\begin{pmatrix}
1 & 0 & 0 & \cdots & 0
\end{pmatrix}^{T}$
by a two-step procedure.  For the first step, the following theorem is needed \cite{constantinescu_schur_1995} \ :

\begin{thm}[Structure of a positive definite block matrix]\label{theo:1}
Let $X$ and $Z$ be positive operators in $\mathcal{L}(\mathcal{H}_{X} )$ and $\mathcal{L}(\mathcal{H}_{Z} )$ respectively. Then the following are equivalent :
\begin{itemize}
\item The operator $A = 
\begin{pmatrix}
X & Y \\
Y^{\ast} & Z
\end{pmatrix}$ is positive
\item
There exists a unique contraction $\Gamma$ in $\mathcal{L}(  \mathcal{R}(Z) , \mathcal{R}(X) )$ such that 
\begin{equation}
Y = X^{1/2}\Gamma Z^{1/2}
\end{equation}
\end{itemize}
\end{thm}

\begin{proof}
Annexe~\ref{annexe1}
\end{proof}
Let us now apply this relation repeatedly on a SPD matrix. To fix ideas, let the  $ 3 \times 3 $ (block-)matrix 
\begin{equation}R =
\begin{pmatrix}
R_{1,1} & R_{1,2} & R_{1,3} \\
R_{1,2}^{\ast} & R_{2,2} & R_{2,3} \\
R_{1,3}^{\ast} & R_{2,3}^{\ast} & R_{3,3}
\end{pmatrix}
\end{equation}
and apply theorem \ref{theo:1} to $ \begin{pmatrix} R_{1,1} & R_{1,2} \\ R^{\ast}_{1,2} & R_{2,2} \end{pmatrix} $, $ \begin{pmatrix} R_{2,2} & R_{2,3} \\ R^{\ast}_{2,3} & R_{3,3} \end{pmatrix} $ and finally to $ \begin{pmatrix} R_{1,2} & R_{1,3} \\ \end{pmatrix} $. Note that when a square root of a (block-)matrix has to be chosen, it is done according to the Schur decomposition given in Annexe~\ref{annexe1}. At each step, a contraction $ \Gamma_{i, j} $ is generated with respect to the indices of the upper and lower (block-)matrices of the main diagonal, \emph{e.g.} $ \Gamma_{1, 2} $ for the first $ \begin{pmatrix} R_{1,1} & R_{1,2} \\ R^{\ast}_{1,2} & R_{2,2} \end{pmatrix} $ (block-)matrix. We obtain thus a one-to-one correspondence between the SPD matrix $R$ and the set of contractions $\left\{ \Gamma_{i,j} \right\}_{i=1,2 \; j=3}$. Regarding the huge work of Constantinescu \cite{constantinescu_schur_1995}, we will called these contractions the Schur-Constantinescu parameters. Considering now unit variance and arbitrary size $n \times n$ for the SPD matrix, allows us to write the correspondance:
\begin{equation}
\label{correspondance correlation matrix parcor matrix}
\begin{pmatrix}
I & R_{1,2} &  &  R_{1,n} \\
R^{\ast}_{1,2} & I & \ddots&   \\ 
& \ddots & \ddots & R_{n-1,n}\\
& & & \\
R_{1,n}^{\ast} &  & R_{n-1,1}^{\ast} & I
\end{pmatrix}
\longleftrightarrow
\begin{pmatrix}
0 & \Gamma_{1,2} & \Gamma_{1,3} &  & \cdots &   \Gamma_{1,n}\\
0 & 0 & \Gamma_{2,3} & \Gamma_{2,4} & \cdots & \Gamma_{2,n}\\
\vdots & & \ddots & \ddots & \ddots & \\
 & &  & &  & \Gamma_{n-2,n}\\
 & & & &0 & \Gamma_{n-1,n}\\
0 & 0 & \cdots & &    & 0
\end{pmatrix}.
\end{equation}

Once (\ref{correspondance correlation matrix parcor matrix}) established, each dilation matrix $ W_{i} $ are build-up as a product of Givens rotations of a sequence of Schur-Constantinescu parameters in the following way 
\begin{equation}
W_{i} = G(\Gamma_{i,i+1}) G(\Gamma_{i,i+2}) \cdots G(\Gamma_{i,j})
\end{equation}
where $ G_{\Gamma_{i,i+l} }$ denote the Givens rotation of $ \Gamma_{i,i+l} $: 
\begin{equation}
G(\Gamma_{i,i+l}) = I \oplus \begin{pmatrix}
\Gamma_{i,i+l} & D_{\Gamma_{i,i+l}^{\ast}} \\
D_{\Gamma_{i,i+l}} & -\Gamma_{i,i+l}^{\ast}
\end{pmatrix} \oplus I.
\end{equation}
When the SPD matrix is Toeplitz, which correspond to a stationary underlying process, then all dilation matrices $ W_{i} $ are identical and they take the form
\begin{equation}
W_{i} = U = 
\begin{pmatrix}
\Gamma_{1} & D_{\Gamma_{1}^{\ast}} \Gamma_{2} & D_{\Gamma_{1}^{\ast}}D_{\Gamma_{2}^{\ast}} \Gamma_{3} & D_{\Gamma_{1}^{\ast}}D_{\Gamma_{2}^{\ast}} D_{\Gamma_{3}^{\ast}}  \Gamma_{4} & \cdots \\
D_{\Gamma_{1}} & - \Gamma_{1}^{\ast}\Gamma_{2}  & - \Gamma_{1}^{\ast}D_{\Gamma_{2}^{\ast}} \Gamma_{3} & - \Gamma_{1}^{\ast}D_{\Gamma_{2}^{\ast}} D_{\Gamma_{3}^{\ast}} \Gamma_{3} & \cdots \\
0 & D_{\Gamma_{2}} & - \Gamma_{2}^{\ast}\Gamma_{3}  & - \Gamma_{2}^{\ast}D_{\Gamma_{3}^{\ast}} \Gamma_{4} & \cdots \\
0 & 0 &  D_{\Gamma_{3}} & - \Gamma_{3}^{\ast}\Gamma_{4} & \cdots \\
0 & 0 & 0 & D_{\Gamma_{4}} & \cdots \\
\cdot & \cdot & \cdot & \cdot &  & \cdots \\
\cdot & \cdot & \cdot & \cdot &  & \cdots \\
\end{pmatrix}
\label{dilation matrix}
\end{equation}

which is nothing less than the Naimark dilation introduced in the first part, \emph{i.e.} $ R_{i,j} = R_{j-1} = [1\ 0\ 0\ \cdots] U^{j-i} [1\ 0\ 0\ \cdots]^{T} $.

For a sake of completeness, we give the correspondence between the coefficients of the SPD matrix (the left-hand side of (\ref{correspondance correlation matrix parcor matrix}) ) and the Schur-Constantinescu parameters:

\begin{thm} 
\label{Schur-Constantinescu parametrization}
The matrix $R^{(n)} = [R_{k,j}]_{k,j=1}^{n}$, satisfying $R_{j,k}^{\ast} = R_{k,j}$ is positive if and only if 
\begin{itemize}
\item $R_{kk} \geqslant 0 $ for all k
\item There exists a family $\lbrace \Gamma_{k,j} \mid k,j =1, \cdots n, k\leqslant j \rbrace$ of contraction such that 

\begin{equation} \label{Schur-Constantinescu formula}
R_{k,j} = B_{k,k}^{\ast} (L_{k,j-1} U_{k+1,j-1} C_{k+1,j}\ +\ D_{\Gamma_{k,k+l}^{\ast}} \cdots D_{\Gamma_{k,j-l}^{\ast}} \Gamma_{k,j} D_{\Gamma_{k+1,j}} \cdots D_{\Gamma_{j-1,j}} )B_{j,j}
\end{equation}

where $B_{k,k}$ is any square root of $R_{k,k}$
\end{itemize}
\end{thm}

Where we have :

\[
L_{k,j} = [ \Gamma_{k,k+1} \quad D_{\Gamma_{k,k+l}^{\ast}}\Gamma_{k,k+2} \quad \cdots \quad D_{\Gamma_{k,k+l}^{\ast}} \cdots D_{\Gamma_{k,j-1}^{\ast}}\Gamma_{k,j} ]
\] 

a row contraction associated o the set of parameters $\lbrace \Gamma_{k,m}  \mid k < m \leq j \rbrace$,

\[
C_{k,j} = [ \Gamma_{j-1,j} \quad \Gamma_{j-2,j}D_{\Gamma_{j-1,j}} \quad \cdots \quad \Gamma_{k,j} D_{\Gamma_{k+1,j}} \cdots D_{\Gamma_{j-1,j}} ]^{T}
\]

a column contraction associated to the set of parameters $\lbrace \Gamma_{m,j}  \mid m = j-1, \cdots k \rbrace$, and finally 
\[ 
U_{k,j} = G(\Gamma_{k,k+1})G(\Gamma_{k,k+2}) \cdots G(\Gamma_{k,k+j}) \left( U_{k+1,j} \oplus I \right)
 \]
 
\begin{proof}
This theorem is proved in \cite{constantinescu_schur_1995}.
\end{proof} 
A different approach leading to the same results can be found in \cite{timotin_prediction_1986}, using directly the Kolmogorov decomposition. In \cite{kailath_naimark_1986} the Naimark dilation is constructed using lattice filter and finally applications of this decomposition in quantum mechanics are to be found in \cite{tseng_contractions_2006,tseng_dilation_2006} for example. 

%++++++++++++++++++++++++++++++++++++++++++++++++++++++++++++++++++

\subsection{Link between the Levinson algorithm and Constantinescu's formulae.}

The theory explained so far provides infinite dimensional operators. Of course, it is not possible to treat with infinite dimensional object in practice, particularly with numerical computations. To strengthen our development, it is then necessary to adapt the theory to the finite dimensional case. This is tackled for example in \cite{foias_geometric_1990} where a Levinson algorithm is performed on block matrices.

\subsubsection{The principles of Levinson Algorithm}
In classical stationary signal processing, the Levinson algorithm is a procedure to estimate linear regression coefficients, say $ \left\{ \alpha_{k} \right\} $ as the projection of the signal at time $t$ onto the space spanned by its $n$ past values:  $ X_{t} = -\sum_{i = 1}^{ n }{ \alpha_{i} X_{t-i} } + \epsilon_{t}^{n}  $, where $ \epsilon $ is a white gaussian noise by assumption. The last coefficient of the model is the $n$-th reflexion (parcor) coefficient. If the correlation matrix of the stationary process is noted $R = \left\{ R_{ij}\right\}_{i,j=1,\ldots,n}$ we have the so called Yule-Waker equation: 
\begin{equation}
\begin{pmatrix}
R_{0} & R_{1} &  \cdots & R_{n}\\
 R_{1}^* & \ddots & \ddots \\
\vdots & \ddots & \ddots & \\
R_{n}^* & \cdots &  & R_{0}
\end{pmatrix}
\begin{pmatrix}
1\\
\alpha_{1}^{n}\\
\vdots\\
\alpha_{n}^{n}
\end{pmatrix} = 
\begin{pmatrix}
\epsilon_{n}^{2}\\
0\\
\vdots\\
0
\end{pmatrix}.
\end{equation}
The Levinson algorithm performs a recursion on the order $n$ of the model which appears to be exactly what is done for the dilation matrices. 
\subsubsection{Analogy between Levinson and Dilation}
Let us recall that the form of the infinite dimensional dilation matrix is given by (\ref{dilation matrix}). This matrix acts on the space $ \mathcal{K} = \mathcal{H} \oplus \mathcal{D}_{1} \oplus \mathcal{D}_{2} \cdots  $ where the set of $ \mathcal{D}_{i} $ stands for the defect spaces. The forward prediction is defined by
\begin{equation}
U^{n} - P_{\mathcal{H} \oplus \mathcal{D}_{1} \oplus \mathcal{D}_{2} \cdots \oplus \mathcal{D}_{n}}U^{n}h = \begin{pmatrix}
0,0,\cdots, 0,\mathcal{D}_{n}\mathcal{D}_{n-1}\cdots \mathcal{D}_{1}h, 0, \cdots
\end{pmatrix}
\end{equation}
for $ h \in \mathcal{H} $ and the non zero elements are in the $ n+1 $ position. In \cite{foias_geometric_1990}, the authors denotes the right-hand part by $ \phi_{n}\epsilon_{n}h $ where $ \epsilon_{n} = \mathcal{D}_{n}\mathcal{D}_{n-1}\cdots \mathcal{D}_{1}  $ and $ \phi_{n} $ embeds onto the $ n+1 $ position. They prove so the exact accordance with the classical Levinson procedure. In that case, $ \epsilon_{n} = \mathcal{D}_{n}\mathcal{D}_{n-1}\cdots \mathcal{D}_{1}  $  is therefore nothing else than the prediction error when the Levinson procedure is applied on the positive-definite kernel representing the correlation matrix. \\
Similarly, for the backward prediction error we have
\begin{equation}
h - P_{\cup{C_{n}}}h = \phi_{n^{\ast}}\epsilon_{n^{\ast}}h
\end{equation}
where $ \epsilon_{n^{\ast}} = \mathcal{D}_{n^{\ast}}\mathcal{D}_{n-1^{\ast}}\cdots \mathcal{D}_{1^{\ast}} $ and $ C_{n} $ stands for the $n$ first columns of $U$.
In terms of dilation, the Levinson procedure at step $n$ can be written as
\begin{equation}
\label{Levinson and dilation}
\epsilon^{\ast}_{n^{\ast}}\Gamma_{n+1}\epsilon_{n}h = R_{n+1} + \sum_{i=1}^{n}{R_{i}A_{n,i-1}h}
\end{equation}
where the $ A_{n,i} $ are the Levinson regression coefficients. Rewrite the equation such that
\begin{equation}
\label{Levinson and dilation}
R_{n+1} = \epsilon^{\ast}_{n^{\ast}}\Gamma_{n+1}\epsilon_{n}h - \sum_{i=1}^{n}{R_{i}A_{n,i-1}h}
\end{equation}

and compared it with theorem \ref{Schur-Constantinescu parametrization}, allows to see that the Levinson recursion is equivalent to the parametrization of positive-definite kernel by means of Schur-Constantinescu parameters.
We remark in addition that the necessary truncation error made when dealing with dilation matrices is now quantified.\\
Dilation matrices being now fully introduced, we focus the attention of the reader on the hidden information contained in their timely geometrical dissemination.

%======================================     SECTION    =========================================================
\section{Analysis of curves on a manifold induced by the dilation}\label{section:2}

%+++++++++++++++++++++++++++++++++++++++++++++++++++++++++

%Starting from a positive definite kernel, we have obtained a set of matrices, the dilation matrices, composed by the partial correlation coefficients (parcors). When the real (respectively complex) process is periodically correlated, its parcors inherit of the periodicity and obviously we obtained a periodic set of dilation matrices.
% These matrices are theoretically operator of infinite dimension but as we dispose of only a finite set of parcors, the theoretical matrices of (\ref{dilation matrix}) are truncated. Matrices of the form of (\ref{dilation matrix}) are a general rotation matrices, reducing their dimension to $n \times n$ leads to a perfect rotation operator which belongs then to SO(n) and SU(n) when dealing with complex processes. 

parcors, composing dilation matrices, have already been given a geometrical point of view, as for exemple in \cite{yang_riemannian_2010} where the sequence of parcors asociated with a stationary process is seen as a point onto the Poincaré polydisk $ \mathcal{P}^{n} $, \emph{i.e.} the product of Poincaré disk. To give geometrical settings, a distance to characterize individual parcors is then proposed and discussed. In \cite{brigant_computing_2016}, a stochastic process is studied under the local stationarity assumption. To each stationary slice of the process corresponds a sequence of parcors, represented as a point in the Poincaré polydisk $ \mathcal{P}^{n} $ as well. A trajectory is then generated on that space which materializes a curve on the manifold $ \mathcal{P}^{n} $. The underlying computations are quite intricate due to the product manifolds  and the question of nonstationarity arises. Based on the work of \cite{brigant_computing_2016,brigant_discrete_2017,celledoni_shape_2015} and \cite{zhang_video_based_2015}, we propose then to give a particular attention to this question. We first make use of Dilation theory introduced in section \ref{section:1}. When the process under study is non-stationary, a set of matrices $W_i$ is obtained. The basic idea for having geometric information on the non-stationary process is therefore to characterize the trajectory formed by the set of dilation matrices. These matrices are theoretically operator of infinite dimension but as we dispose of only a finite set of parcors, the theoretical matrices of (\ref{dilation matrix}) are truncated. Matrices respecting (\ref{dilation matrix}) are general rotation matrices which becomes perfect rotation operator belonging then to $SO(n)$ for real processes and $SU(n)$ when dealing with complex processes, when their dimension are reduced to $n \times n$. Our aim is finally to analyse those curves living on Lie Group of rotation matrices and emphasize the geometry or more precisely the intrinsic geometry formulation of these objects. For exemple, we aim at comparing different curves coming from different processes or at resuming many realizations of a stochastic process (multiple measurements) through the computation of the mean of the associated several curves. The question of the computation complexity still rises but many results have been proposed recently to overcome this difficulty and to propose closed form formulations. In particular, it is predicated to extract the shape of the trajectory for it contains the essentials, in topologic sense, information.\\ 
To allow the curves comparison, we have based our development on the work of \cite{brigant_computing_2016} and \cite{celledoni_shape_2015}. First we define the manifold $ \mathcal{M} $ given by the set of all curves in the based manifold. That leads to an other space, the shape space, for which the manifold $ \mathcal{M} $ will be a fiber bundle. We dispose then of a metric in $ \mathcal{M} $ from which a metric on the shape space is deduced. These steps are now explained in the followings.

\subsection{Basic Outline of Geometry}

Curves of interest are those ones living in the Lie group of real rotation matrices, this yields to  $ c : [0,1] \rightarrow SO(n)$. For sake  of clarity, we suppose that $c$ is continuous, we will come back to the case of discrete curves later. To study geometrical features of such curves, we interest ourself in the set of all curves lying in $ SO(n) $ ( where $ SO(n) $ is seen as a manifold) with non-vanishing velocity, \emph{i.e.} $ \mathcal{M} = \left\{ c \in \mathcal{C}^{\infty}([0,1],SO(n)) : c'(t) \neq 0 \; \forall t \right\} $, this is in fact a sub-manifold of $ \mathcal{C}^{\infty}([0,1],SO(n)) $. A curve $c$ is so a particular point in $ \mathcal{M} $. The tangent space at a curve $c$ is given by
\begin{equation}
 T_{c}\mathcal{M} = \left\{ v \in \mathcal{C}^{\infty}([0,1], TSO(n) ) : v(t) \in T_{c(t)}SO(n)  \right\}
 \end{equation}

where $ TSO(n) $ denote the tangent bundle of the based manifold $ SO(n) $. Note that a tangent vector is a curve in the tangent space of $ SO(n) $. In this manifold, the expression of distances and thus geodesics depends on the chosen metric. When comparing two curves, it is natural that the distance between this two curves should remain the same if the curve are only reparametrized, \emph{i.e.} if we define other curves that passes through the same points than the original curves but at different speeds. When the curve is discretized as we will see in the sequel, doing a reparametrization is equivalent to change the chosen points (see Figure \ref{fig:reparam}). A reparametrization is represented by an increasing diffeomorphism $ \phi \in \mathcal{D} : [0,1] \rightarrow [0,1] $ acting on the right of the curve by composition. In other words, we required that the Riemannian metric $g$ on $ \mathcal{M} $ satisfies the following property

\begin{equation}
g_{c\circ \phi}(u\circ \phi, v \circ \phi) = g_{c}(u,v)
\end{equation}
for all $ c \in \mathcal{M} $, $ u,v \in T_{c}\mathcal{M} $ and $ \phi \in \mathcal{D} $.

\begin{figure}[!h]
\center
\includegraphics[width = 10cm]{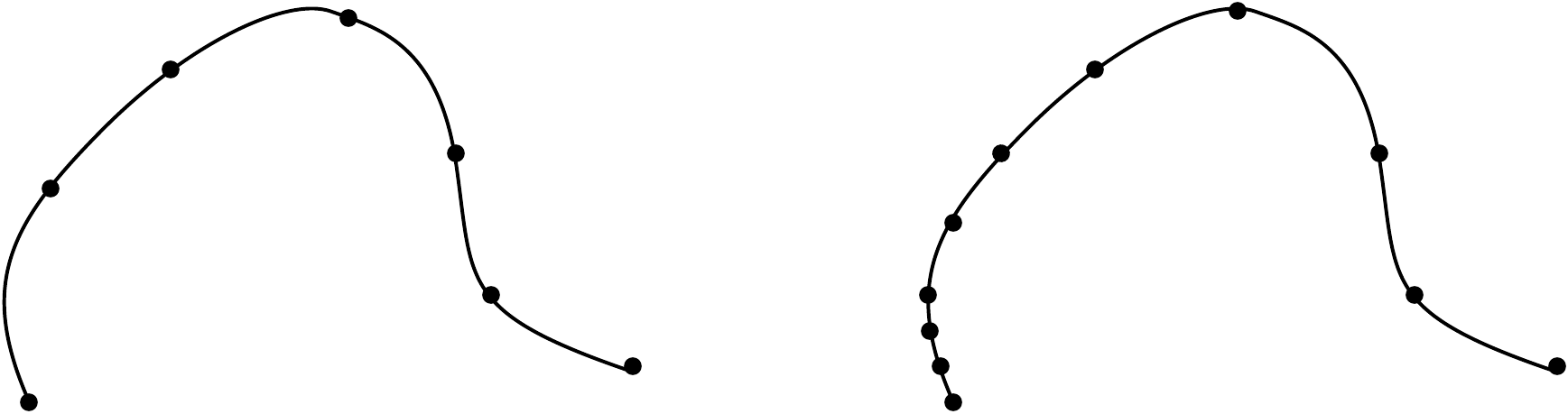}
\caption{example of a reparametrization of a curve. Here, it consists in changing the discretization. }\label{fig:reparam}
\end{figure}

This property is called reparametrization invariance. We insist on the fact that $g$ is the metric on $ \mathcal{M} $, the space of all curves on $ SO(n) $ and not on $SO(n)$ itself. In terms of distances, this gives

\begin{equation}
d_{\mathcal{M}}(c_{0} \circ \phi, c_{1} \circ \phi) = d_{\mathcal{M}}(c_{0}, c_{1})
\end{equation}
where $ d_{\mathcal{M}} $ denote the distance on $ \mathcal{M} $ corresponding to the metric $g$.
The reparametrization introduced above induces an equivalence relation between points in $ \mathcal{M} $ in such a way we have
\begin{equation}
c_{0} \sim c_{1}  \iff \exists \phi \in \mathcal{D} : c_{0} = c_{1}\circ \phi.
\end{equation}

With this equivalence relation, a quotient space can be constructed as the collection of equivalence classes, it is named the shape space and has the following writting

\begin{equation}
\mathcal{S} = \mathcal{M} / \sim,\ or\ \mathcal{S} = \mathcal{M} / \mathcal{D}.
\end{equation}

A distance function on the shape space is obtained from the distance on $ \mathcal{M} $ as follows

\begin{equation}
d_{\mathcal{S}}([c_{0}],[c_{1}]) = \underset{\phi \in \mathcal{D}}{inf} d_{\mathcal{M}}(c_{0},c_{1} \circ \phi)
\end{equation}

where $ [c_{0}], [c_{1}] $ are representatives of the equivalence classes of $ c_{0} $ and $ c_{1} $ respectively. It can be shown that this distance is independent of the choice of the representatives. It is in fact inherited from the fiber bundle structure $ \pi = \mathcal{M} \rightarrow \mathcal{S} $. As closed curves are of main interest in this work, we can also define the set
\begin{equation}
\mathcal{M}^{c} = \left\{ c \in \mathcal{C}([0,1], SO(n)) : c'(t) \neq 0, c(0) = c(1) \right\}.
\end{equation}

Basically, the closure of a curve just imposes the equality of the first and the last point of it, and not of their first derivative. Consequently, $ \mathcal{M}^{C} $ turns into  

\begin{equation}
\mathcal{M}^{c +} = \left\{ c \in \mathcal{C}([0,1], SO(n)) : c'(t) \neq 0, c(0) = c(1), c'(0) = c'(1) \right\} . 
 \end{equation}

We need now to introduce the Square Root Velocity function (SRV function) \cite{srivastava_shape_2011}, in which a curve is represented by its starting point and its normalized velocity at each time $t$. There are several possibilities to define the SRV of a curve. The more general definition is the following

\begin{equation}
\begin{aligned}
F  : \mathcal{M} & \rightarrow SO(n) \times T\mathcal{M}\\
c & \rightarrow \left( c(0), q = \dfrac{c'}{\sqrt{ \norm{c'}}} \right)
\end{aligned}
\end{equation}

 But we can go further and benefit from the specific case that we are dealing with : the data are in a Lie group, $ G = SO(n) $, in this section, we will denote the base manifold $G$ to emphasize its group structure, and $g$ is an element of the group. Recall that a Lie group G is a set that is both a smooth manifold and a group, that is this is a manifold in which the group multiplication $ G \times G \rightarrow G, (x,y) \mapsto xy $ and the inversion map, $ g \rightarrow G, x \mapsto x^{-1} $ are smooth.

Therefore, we can expect a simplification on the previous expressions. The main strength of a Lie group G is that it can be approximated by its Lie algebra $ \mathfrak{g}  $, or in other words, the Lie algebra is the best \emph{linear approximation} of a Lie group. The Lie algebra is homeomorphic to the tangent space at the identity. As in \cite{celledoni_shape_2015} we consider only curve that start at the identity: this is because other curve can be reduced to this case by right or left translation. In this settings it is interesting to turn the SRV function into the transported SRV function (TSRV). This is basically the SRV that has been parallel transport to a reference point. Different version has been given, as it can be seen in \cite{bauer_constructing_2014}, \cite{celledoni_shape_2015}, \cite{zhang_video_based_2015} . They differ by the choice of their reference point. As we study curve in a Lie group, the identity (the starting point of our curve in our case) is a particularly good choice as a reference point. In a Lie group a parallel transport operation can be defined, here again, by the right (or left) translation. This justify that we can take, as suggested in (\cite{celledoni_shape_2015})  a TSRV function of the form:

\begin{equation}
\begin{aligned}
F_{Lie} : \mathcal{C}^{\infty}([0,1], G ) & \longrightarrow  SO(n) \times \left\{ q \in \mathcal{C}^{\infty}([0,1], \mathfrak{g}), q(t) \neq 0,\ \forall t \in [0,1] \right\} \\
  F_{Lie}(c)(t) & = ( c(0), q(t)) =   \left( c(0), \dfrac{R^{-1}_{c(t) \ast}(c'(t))}{\sqrt{ \norm{ c'(t) } }} \right) = \left( c(0), \dfrac{T_{c}^{c(t) \rightarrow I}(c'(t))}{\sqrt{ \norm{ c'(t) } }} \right)
\end{aligned}
\end{equation}

where R is the right translation on the group, $ R_{g_{1}}(g_{2}) = g_{2} g_{1} $, $ R_{g \ast} = T_{e} R_{g} $ is the tangent map (the derivative equivalent on manifolds, equivalent notation : $ dR_{g} $) at the identity, and $ \norm{\cdot} $  is a norm induced by a right-invariant metric on $G$, and $ T_{c}^{c(t) \rightarrow I} $ denotes the parallel transport from $ c(t) $ to the identity according to the curve $c$. So a curve is now represented as an element of the tangent bundle $ \left(  c(0), q(t) \right) \in M \times TM $ (recall that $q$ draw a curve in the tangent bundle), and $ c(0) $ is the identity element of the Lie group. As mentioned earlier, the TSRV function is the SRV function that has been transported to a certain point, in other words, we use the natural reference point of a Lie group that consists in its identity element to define the parallel transportation, and we apply it to the tangent curve.
The inverse of the SRV function is then straightforward: for every $ q \in \mathcal{C}^{\infty}([0,1], T\mathcal{M}) $, there exists a unique curve $c$ such that $ F(c) = q $. We have $ c(t) = \int_{0}^{t}{ q(r) \norm{q(r)}dr } $ where $ \norm{\cdot} $ is the norm in $ SO(n) $.

%+++++++++++++++++++++++++++++++++++++++++++++++++++++++++
\subsection{Metric and distance over $\mathcal{M}$ and $\mathcal{S}$}

We now give insights on a relevant metric that should be used on $ \mathcal{M} $ to compare different closed trajectories. The idea is to have a metric on $ \mathcal{M} $ that induced a "coherent" distance on the shape space $ \mathcal{S} $. The following development and expression of metrics and distance can be found in \cite{brigant_computing_2016}.  The distance on the shape space is used to compare how to curve are intrinsically different. It has been in \cite{michor_vanishing_2005} that the simple $ L^{2} $ metric on $\mathcal{M} $ given by

 \[
 g_{c}^{L^{2}}(u,v) = \int{\inner{u}{v} \norm{c'(t)} dt}
 \]
 where $ \inner{\cdot}{\cdot} $ is the Riemannian metric on $ SO(n) $ induced a vanishing metric on the shape space, that is we can not differentiate shape with this metric. Another type of metric has been investigated, which induced non-vanishing metrics on the shape space: the idea is to add higher order derivatives.  These are the \emph{elastic metric}, derived from the Sobolev metric \cite{bauer_why_2016}, \cite{gelman_bayesian_2014}, which in the case of an Euclidean space $ \mathbb{R}^{n} $: 
 
\begin{equation}
g_{c}^{a,b}(u,v) = \int{ \left(  a^{2}\inner{ D_{l}u^{N} }{ D_{l}v^{N} } + b^{2}\inner{ D_{l}u^{T} }{D_{l}v^{T}} \right) \norm{c'(t)}dt }
\end{equation}

Where: $ \inner{\cdot}{\cdot} $ is the Euclidean canonical metric on the based manifold ($ SO(n) $ in our case), $ D_{l}u = h'/\norm{c'} $, $ D_{l}u^{T} = \inner{D_{l}u}{w}w $, with $ w = c'/\norm{c'} $ and $ D_{l}u^{N} = D_{l}u - D_{l}u^{T} $ this way, $ \left( D_{l}u^{N}, D_{l}u^{N} \right) $ defines a mobile frame along the curve c, see figure (\ref{mobile frame-measure}).
 
Here we are only interesting on the special metric thats has been proposed in \cite{brigant_computing_2016}, and which is an adaptation of the elastic metric for Riemannian manifold: 

\begin{equation}
g_{c}(u,v) = \inner{u(0)}{v(0)} + \int{ \left(  \inner{ \nabla_{l}u^{N} }{ \nabla_{l}v^{N} } + \dfrac{1}{4}\inner{ \nabla_{l}u^{T} }{\nabla_{l}v^{T}} \right) \norm{(c't)}dt }
\end{equation}
 
Where $ \inner{\cdot}{\cdot} $ denotes the Riemannian metric of the original space $SO(n)$, $ \nabla $ is the Levi-Civita connection that corresponds to $ \inner{\cdot}{\cdot} $; $ \nabla_{l}u = \dfrac{1}{\norm{c'} }\nabla_{c'}h $, $ \nabla_{l}u^{T} = \inner{\nabla_{l}u}{w}w $, $ w = c'/\norm{c'} $. The computations being done now in a manifold space, the Levi-Civita connection replaced the ordinary derivative of $ \mathbb{R}^{n} $.

\begin{figure}
\label{decomposition tangent normal }
\center
\includegraphics[width = 15cm]{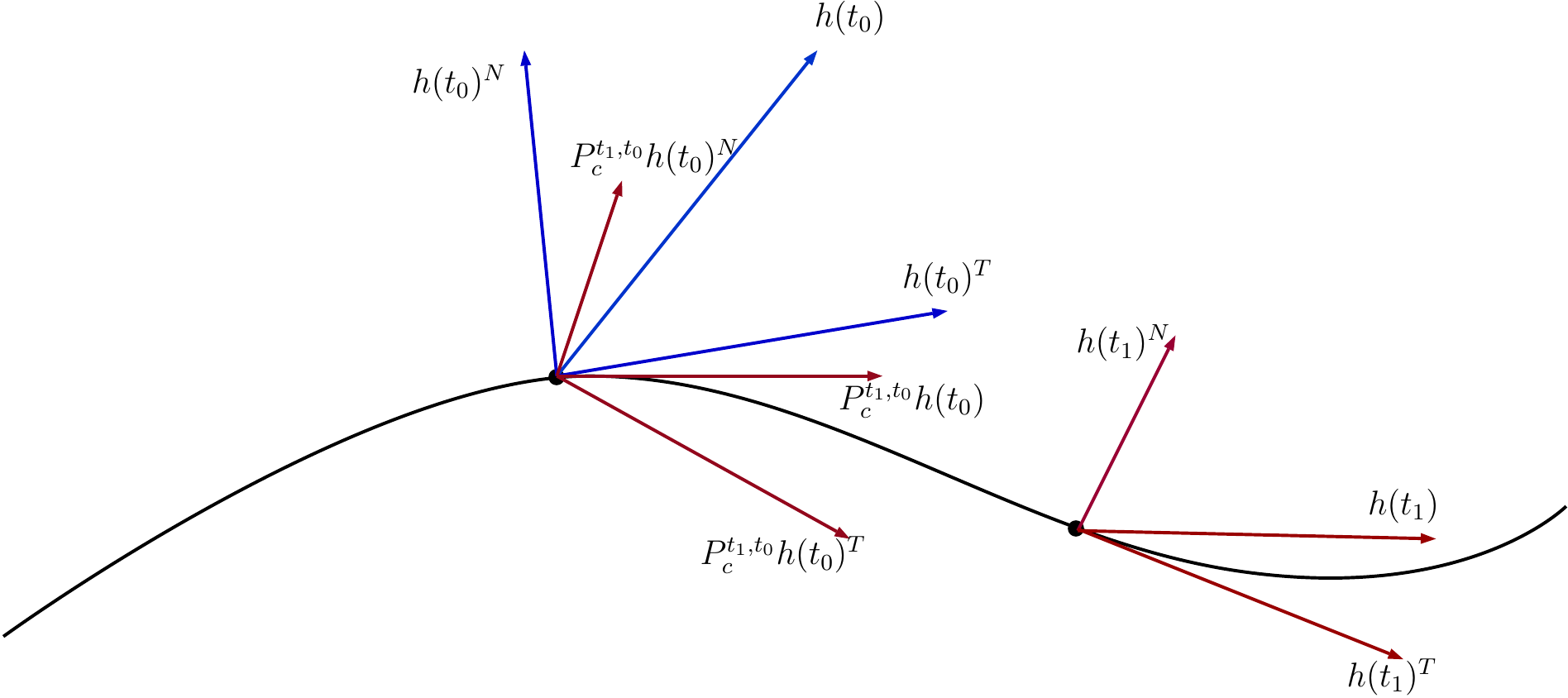}
\caption{The inner product measures the angle between a frame at a given point and the parallel transport version of this frame at a latter time.}
\label{mobile frame-measure}
\end{figure}

This metric is obtained as the pullback of the canonical $ L^{2} $ metric in the SRV framework. Recall that if we have a map $ \phi : \mathcal{U} \rightarrow \mathcal{V} $ and a one-form (that is a section) $ \alpha $ of $ T^{\ast}\mathcal{V} $, the cotangent bundle of $ \mathcal{V} $, the pullback of $ \alpha $ is the one-form \emph{on $ \mathcal{U} $} defined by : $ \left( \phi^{\ast}\alpha \right)_{x}(X) = \alpha_{\phi(x)}\left( d \phi(x) (X) \right) $, for all $ x \in \mathcal{U}, \ X\in T_{x}\mathcal{U} $. Here, $ \alpha  = F$, and $ \mathcal{U} = \mathcal{M}, \mathcal{V} =  T\mathcal{M} $ and as we the ambiant space is a Lie group we see that the metric is the pullback of the $ L^{2} $-metric on the  tangent bundle (\emph{i.e.} an inner product on the double tangent bundle TTM) defined by
\begin{equation}
\overset{\sim}{g}_{c}(f,g) = \int_{0}^{1}{  \inner{ \nabla_{s}q^{f}(0,t) }{ \nabla_{s}q^{g}(0,t) }dt  }
\end{equation}

where for any $ f \in T_{c} $ a path of curve $ s \mapsto c_{f}(s) \in \mathcal{M} $ is defined, as usually, such that $ c_{f}(0) = c, \ \left( \partial c_{f}/ \partial s \right)(0) = f  $ (recall that a tangent f vector to a point c can be seen as an equivalence class of  infinitesimal deformations of a path that passes through c at time 0). $ q^{f} $ is the TSRV function of $ c_{f} $ . Due to the transportation of the TSRV, the starting point is no longer required in the TSRV domain. Thus, from the definition of a metric on $ T\mathcal{M} $, that apply on the space of TSRV functions we can deduce a metric on the space of curves as $ g_{c}(u,v) = \overset{\sim}{g}_{F(c)}( T_{c}F(u), T_{c}(F(v) ) $, where $ T_{c}R $ denote the tangent map (the derivative) of the TSRV function at c (which carry the information about the starting point of the curve).

What is particular with this metric compared to the classical elastic metric, is that the starting point of the curves intervene explicitly, which is quite important when we deal with periodic curves, as explained before. It could be possible to add a term such that $ \inner{c'_{0}(0)}{c'_{1}(0)}_{TSO(n)} $ for some metric on $ SO(n) $ to underline the fact that curves are closed even in their first order derivatives.

One can notice that the SRV  formulation of the metric takes a quite simple form. In this framework, the length of a path of curve (notice that it is the length of a path of curve and not the length of a curve in $SO(n)$) is 

\begin{equation}
L(c) = \int_{0}^{1}{ \sqrt{  \norm{x(s)}^{2} + \int_{0}^{1}{ \norm{ \nabla_{\partial c / \partial s } q(s,t)}^{2}dt }} ds}
\end{equation}

Once the geometry has been settled in $\mathcal{M}$, the geometry of the state-space can be derived from its quotient structure (recall that the space of curve induced a fiber bundle structure above the state-space). Before we have to remember the definition of the decomposition of the tangent bundle as an horizontal and a vertical tangent bundle : given the fiber bundle $ \pi : \mathcal{M} \rightarrow \mathcal{S} $ the tangent bundle can be decomposed into a vertical and a horizontal subspace : 

\begin{equation}
T\mathcal{M} = \mathcal{H}_{\mathcal{M}} \oplus \mathcal{V}_{\mathcal{M}}
\end{equation}

with $  \mathcal{V}_{\mathcal{M}} = ker \left( T_{c} \pi \right)  $ ( $ T_{c} $ is the tangent map), and $ \mathcal{H}_{\mathcal{M}} = \left( \mathcal{V}_{\mathcal{M}} \right) ^{\bot} $ where the orthogonality has to be understood in the sense of the metric defined earlier. This metric is reparametrization invariant, that is constant along the fibers, so it can be proven that the horizontal space $ \mathcal{H}_{\mathcal{M}} $ (which concern the fibers) and the tangent space in the shape-space are \emph{isometric} : $ \pi  $ is a \emph{Riemannian submersion}. Translated in terms of metrics, we have : 

\[ 
g_{c}({u_{\mathcal{H}}, v_{\mathcal{H}}}) = [g]_{\pi(c)} \left(  T_{c}\pi (u), T_{c}\pi (v) \right) 
 \]
where $ [g] $ denote the metric \emph{on the shape space}.

\begin{figure}
\center
\includegraphics[scale = 0.5]{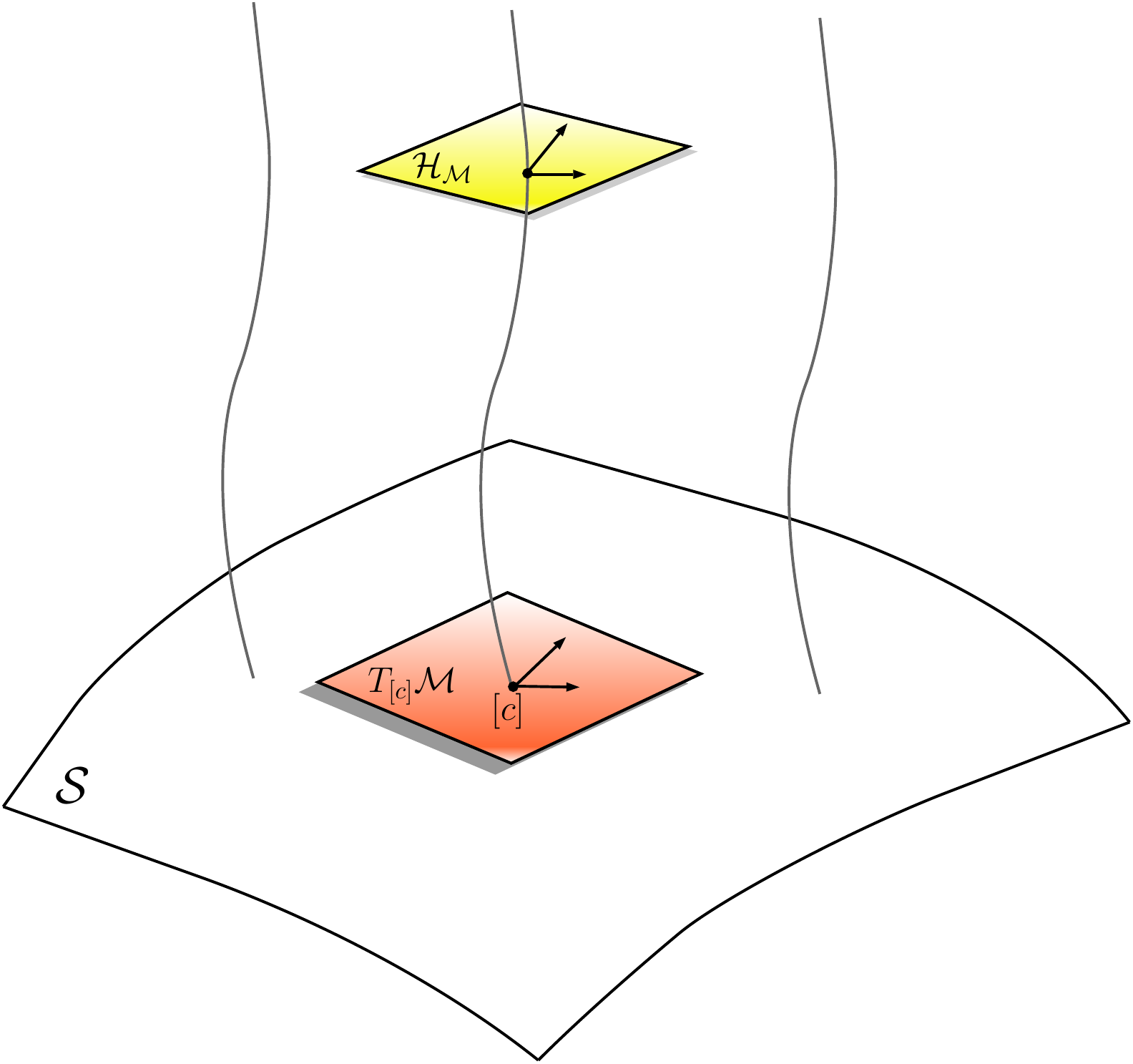}
\caption{The tangent space $T_{[c]}\mathcal{M}$ at a point $[c]$ in the shape space $\mathcal{S}$  is isomorphic to the horizontal part $\mathcal{H}_{\mathcal{M}}$ of the tangent space at a point on the associated fiber. }
\label{fiber_bundle structure}
\end{figure}

Similar result in a different (but still close) context is used in \cite{yili_li_riemannian_2013}, lemma 1. In terms of distances, this can be understood in the following sense:

% question : en comparant avec le lemme 1 du papier LI-Riemmanian distances for signal classification by power spectral density.pdf, quel serai l'equivalent ici de la formule de la metrique avec la trace, sachant que la fibre est l'espace des courbes dand notre cas et la manifold es l'espace des formes et non un espace de matrices comme dans le papier.... la formule avec la trace correspond elle a une forme de killing.... ? 

 The geodesic $ s \mapsto [c](s) $ between $ [c_{0}] $ and $ [c_{1}] $ in the state-space is the projection of the horizontal geodesic linking $ c_{0} $ to the fiber containing $ c_{1} $. In fact the horizontal geodesic between $ c_{0} $ and $ c_{1} $ intersect the fiber of $ c_{1} $ at  the a re-parametrized version of $ c_{1} $, $ c_{1} \circ \phi $ which give the distance in the shape space :

\begin{equation}
[d]([c_{0}],[c_{1}]) = d_{g}(c_{0}, c_{1} \circ \phi)
\end{equation}
where $ [d] $ denote the distance in $ \mathcal{S} $, and $ d_{g} $ the distance on the space of curve induced by the aforementioned riemannian metric.
Here again a recursive procedure \cite{brigant_computing_2016} enable the computation of the horizontal geodesic. In the TSRV function formulation, the distance problem is equivalent to an optimization problem : 

\begin{equation}
\label{distance between shapes}
[d]([c_{0}],[c_{1}]) = \underset{\phi \in \mathcal{D}}{inf}\left(  \int_{0}^{1}{  \norm{  q_{0}(t) - q_{1}(\phi(t))\sqrt{\phi'(t)} }^{2}  } \right)^{1/2}
\end{equation}
where $ q_{i}(\cdot) = F(c_{i})(\cdot) $. As these curves in the tangent space had been transported according to the same point, which is the identity element, they lived in the same space and can be compared. In \cite{celledoni_shape_2015} the author used a dynamic linear programming to solve this linear optimization problem, they construct a piecewise linear approximation of $ \phi $. A traditional gradient descent algorithm is also possible.\\

Last but not least, we have to mentioned that in a practical situation, the above formula have to be discretized. This is the object of \cite{brigant_discrete_2017}. Formulas are essentially similar, but in this setting, a curve is now represented by a set of points $ c_{disc}(x_{0}, x_{1}, \cdots, x_{n}) $ and the tangent space turns into

\[ 
T_{disc}\mathcal{M} = \left\{ v = (v_{0}, v_{1}, \cdots,  v_{n}), v_{i} \in T_{x_{i}}SO(n),\ \forall i \right\}.
 \]

Concerning the metric on the space of curve, it becomes

\begin{equation}
\label{discrete metric on curve space}
g_{c_{disc}}(u,v) = \inner{u_{0}}{v_{0}} + \dfrac{1}{n} \sum_{i = 0}^{n-1}{ \inner{ \nabla_{ \partial c / \partial s }q^{u}\left(0, \dfrac{k}{n}\right) }{ \nabla_{ \partial c / \partial s }q^{v}\left(0, \dfrac{k}{n}\right) }  } \quad \forall u,v \in T_{disc}\mathcal{M}
\end{equation}

where, as before, for a $ u \in T_{c_{disc}}\mathcal{M} $, we define a path of piecewise geodesic curves $ (s,t) \mapsto c^{u}(s,t) $ such that the following traditional initial condition are fulfilled, \emph{i.e.}
\begin{eqnarray*} 
c^{u}\left( 0,\dfrac{k}{n} \right) & = & x_{k}, \text{  and  } \\
\left( \partial c^{u} / \partial t \right) \left( 0,\dfrac{k}{n} \right) & = & n\log_{x_{k}}(x_{k+1}).
\end{eqnarray*}
 This is the discrete analogue of the tangent vector of a continuous curve at time $t$. The $ \log$ function is the inverse of the exponential map on the base manifold, $SO(n)$ for us, and here $ c^{u}\left(s,\cdot \right) $ must be a geodesic (on $SO(n)$) between $ x_{k/n} $ and $ x_{(k+1)/n} $. Thus the SRV functions that appears in the formula refers to the SRV function of the piecewise geodesics $ c^{u}\left(s,\cdot \right) $. Then, the discretized version of the SRV function, $ q_{k} = \sqrt{n} \; \log_{x_{k}}( x_{k+1}) / \sqrt{ \norm{ \log_{x_{k}}(x_{k+1}) } } $ is such that 
\begin{equation}
\nabla_{ \partial c / \partial s} q\left( s, \dfrac{k}{n} \right)  = \nabla_{ \partial c / \partial s} q_{k}(s)
\end{equation}

%+++++++++++++++++++++++++++++++++++++++++++++++++++++++++

\subsection{The geodesic equation in the Lie group case}

Before giving the geodesic equation in the space of curves on a Lie group, we start with some preliminaries. We recall some useful facts about Lie group and Lie algebra, for those who are not familiar with these object.\\

A metric $ \inner{\cdot}{\cdot} $ on a Lie group is said to be left invariant if :

\[ 
\inner{u}{v}_{b} = \inner{(dL_{a})_{b}u }{ (dL_{a})_{b}v }_{ab}
 \]
where $ (dL_{a})_{b} $ is the derivative in the manifold field sense (so the tangent map) of the left translation $ L_{a} $ at $b$. A left invariant metric give the same number whenever the vectors are translated on the left. It is straightforward to adapt this definition to a right-invariant metric. A metric that is both left and right-invariant is called a bi-invariant metric. A Lie group endowed with a bi-invariant metric has plenty of import properties that we are can be exploited for our study to curve on shape spaces. We list some of them in the following

\begin{itemize}
\item
The geodesics through $e$ (the identity element) are the integral curves $ t \mapsto exp(tu),\ u \in \mathfrak{g}  $,i.e. the one-parameter groups. Also, since left and right are isometries and isometries maps geodesics to geodesics, the geodesics through any point $ a \in G $ are the left (right) translates of the geodesics through e

\[ 
\gamma(t) = L_{a} \left( exp(tu) \right),\ u \in \mathfrak{g}.
\] 
 Of course we have
\[ 
\gamma'(0) = \left( dL_{a} \right)e(u).
\]
\item
The Levi-Civita connection is given by : $ \nabla_{X}Y = \dfrac{1}{2} [X,Y] , \ \forall X,Y \in \mathfrak{g} $
\item
The curvature tensor is given by : $ R(u,v)w = \dfrac{1}{4}[[u,v],w] $   
\end{itemize}

where $ [\cdot,\cdot] $ denotes the Lie bracket. We can now link these formulas to our based manifold $SO(n)$. \\
A Killing form, $B$, of a Lie algebra is a symmetric  bilinear form $ B : \mathfrak{g} \times \mathfrak{g} \longrightarrow \mathbb{C} $ given by $B(u,v) = tr(ad(u) \circ ad(v))$, where $tr$ denote the trace operator and $ad$ the adjoint representation of the group, namely the map $ ad : G \longrightarrow GL(\mathfrak{g}) $ such that, for all $ a \in G $  $ ad_{a} : \mathfrak{g} \longrightarrow \mathfrak{g} $ is the \emph{linear isomorphism} defined by  $ ad_{a} = d(R_{a}^{-1} \circ L_{a})_{e} $. If now we assume $B$ to be negative-definite, then -$B$ is an inner product and is adjoint-invariant. So, it is a classical result of the Lie theory that -$B$ induces a bi-invariant metric on $G$. Furthermore, the Ricci curvature is given by $ Ric(u,v) = -\dfrac{1}{4}B(u,v) $. 

The Lie algebra of $ SO(n) $ is the set of skew-symmetric matrices, that is matrices $M$ which verifies $ M^{T} = -M $.
The Killing form on $ SO(n) $ is given by : $ B_{\mathfrak{so}(n)} = (n-2)tr(XY) $, and due to the skew-symmetry, we have $ -B_{\mathfrak{so}(n)} = (n-2)tr(XY^{T}) $. Therefore, it induces a bi-invariant metric and the previous formula can be plugged in expression of the metric on the space of curve or on the shape space. To conclude these preliminaries, we see that due to the simpler form of the parallel transportation and of the metric, the distance equations (\ref{distance between shapes}) are now easier to handle. \\

It is now time to give the geodesic equation, relative to our chosen measure. Due to the TSRV, the geodesic equation takes a much simpler form than what can be found in \cite{brigant_computing_2016} and \cite{brigant_discrete_2017}. The formula can be found in \cite{celledoni_shape_2015}. For the sake of completeness, we give a reformulated proof in Annex \ref{AnnexB}. Recall that a geodesic is a particular path of curves. A path of curve is a continuous set of curve $ s \mapsto c(s, \cdot) $ such that for each $s$, $ c(s,\cdot) $ is a point in $ \mathcal{M} $, or equivalently a curve in $ M $, see figure (\ref{path of curves}). So, for each curve of the path of curves we can defined its TSRV function, then for all $ s\in [0,1] $ we have (we omit the letter '$s$' for clarity): $ q = \dfrac{  \partial c / \partial t }{ \sqrt{ \norm{ \partial c / \partial t  } } } $
\begin{thm}
A path of curves  $ [0,1] \ni s \mapsto \left( c(s,0) , q(s,t) \right)$ (t is the parameter of the curve $ c(s,\cdot) $)is a geodesic on $ \mathcal{M} $ if and only if
\begin{equation}
\nabla_{\partial c / \partial s  } \left( \nabla_{\partial c / \partial s } q(s,t) \right) (s,t)  = 0 \quad \forall s,t
\end{equation}
 \end{thm}

\begin{proof}
Annex \ref{AnnexB}
\end{proof}
Thus, we have a quite familiar expression for the geodesic interpolation between two curves $ c_{0} $ and $ c_{1} $, expressed in their TSRV domain:
\begin{equation}
F_{Lie}^{-1} \left( (1-s)F_{Lie}(c_{0}) + s F_{Lie}(c_{1}) \right)
\end{equation}
for $ s \in [0,1] $. This expression is nothing less than a linear interpolation on the transported tangent spaces. \\

Finally we can give a slightly modified version of the algorithm of geodesic shooting. As we have seen for exemple in equation (\ref{discrete metric on curve space}) that a piecewise (continuous) geodesic is required to express the metric. This is because the expression are formally constructed with initial conditions (initial point and initial tangent vector). In \cite{celledoni_shape_2015} an exemple is given for curves on SO(3). We give here a simplified version of their construction. Starting from a continuous curve, they discretize it as $ \left\{ x_{0}, \cdots x_{n} \right\} $ and then construct a piecewise geodesic c between these points: 

\begin{equation}
c(t) = \sum_{k = 0}^{n-1}{ \chi_{k,k+1}(t) exp \left( (t-k) log_{c_{k}}(c_{k+1}) \right)c_{k} }
\end{equation}

Notice that this geodesic is in the base manifold, and not in the space of curve. Thus, the geodesics are expressed in terms of one-parameter groups. But in our case, we already have a discrete curve. The problem is the number of points that we have. Results exists that show the convergence of the discrete case to the continuous case, but when we deal with positive definite matrice with a certain periodicity, we may have not enough matrices. So, a step has to be added, which is another interpolation step. But in order to have a curve at least $ \mathcal{C}^{1} $, and also because distance between the  $ W $ matrices can be quite high, we interpolate first by spline. There are many way to thought spline interpolation on manifold, but one of the simplest is to interpolate in the tangent space, which is Euclidean, and to go back to the manifold via the exponential map. \\

Therefore, our procedure to compare curves on manifold arising from positive definite matrices is the following : 

\begin{enumerate}
\item
$\textbf{Input}$ : a set of rotation matrices $\left\{ W_{i} \right\}_{i}$ , seen as a partially observation of a periodic trajectory on $SO(n)$.
\item
Interpolate with splines between matrices $ W_{i} $.
\item
Truncate the previous curve to obtain a sequence of points $ \left\{ x_{0}, \cdots, x_{n} \right\} $ with a finer discretization.
\item
Interpolate between these points to obtain a piecewise geodesic curve \emph{in the base manifold} $SO(n)$.
\item
Apply geodesic shooting to obtain a geodesic path\emph{ between two curves}, and eventually between the shapes of the two curves.
\item
$ \textbf{Output} $ : distance between two curve in the manifold defined by the set of curves in $SO(n)$
\end{enumerate}

Formally, we have to add a preliminary step when one deal's with curve that do not starts at the identity. We can either compute the TSRV representation of the curve and then transport this representation to the Lie algebra, or first translate the curve to a curve that starts at the identity. 
To end this part, we mention that in the general framework of the dilation, we deal with operators of infinite size. These operators are unitary. So, computations are to be adapted: and the inner product on the tangent space is now derived from the Hilbert Schmidt norm : 

\begin{equation}
\norm{A}_{H.S} = tr(AA^{\ast}) = \sum_{k = 1}^{+\infty}{ \lambda_{k}(AA^{\ast}) }
\end{equation}
 where A is a trace class operator, and $ \left\{  \lambda_{k}(AA^{\ast}) \right\}_{k} $ is the eigenvalues of $ AA^{\ast}$. This is a generalization of the Frobenius norm to the infinite dimensional case. \\
As these operator give rise to the spectral measure of the underlying process, we therefore have a way to compare spectral measures, in a geometrical way that is different from the "classical" use of Fisher information metric or Kullback-Leibler distance in the information geometry context.

%+======================+====================== CONCLUSION   +======================+======================

\section{Conclusion}\label{section:conclu}

We have introduced a new vision of stochastic processes through the geometry induced by the dilation. The dilation matrices of a given processes are obtained by a composition of rotations whose angle correspond to the well-known reflexion coefficients. The advantage to work with these particular matrices is that they are strongly related to the stochastic measure of the process, and thus to its spectra. Furthermore, the dilation theory is independent of the stationarity of the underlying process: when the signal is stationary, its dilation operator is related to the Naimark dilation whereas when the signal is non-stationary, a set of dilation matrices is obtained and it is related to the Kolmogorov decomposition. Rigorously, dilation matrices are infinite dimensional, though we turn them into rotation matrices by truncation. Each of them belongs so to the Special Orthogonal Group $SO(n)$ or the Special Unitary Group $SU(n)$ depending on the real or complex-valued process under study. We focus our attention on the Periodically Correlated (PC) class of non-stationary processes for which a timely ordered set of dilation matrices describes the process measure. This set draws a closed curve on the Lie group of rotation matrices, and describing or classifying the different PC processes is made by curves comparison. We use for that the Square Root Velocity (SRV) function which represents a curve by its starting point and by its normed velocity vector on the space or curves. The metric in the space of curve naturally extent to the space of shapes. It is then possible to compare the shape of curves when the metric is translated to the Lie algebra, achieving therefore a closed-form expression and easy computation.

%+++++++++++++++++++++++++++++++++++++++++++++++++++++++++++++++++

% ========================================================    APPENDIX     ====================================================

\appendix

\section{Defect operator, elementary rotation}\label{annexe1}

Introducing the \emph{defect operator} of a contraction $T$ as being $D_{T} = (I-T^{\ast}T)^{1/2}$, we have the following factorization :

\begin{equation}
\begin{pmatrix}
X & Y \\
Y^{\ast} & Z
\end{pmatrix} =
\begin{pmatrix}
X^{1/2} & 0\\
Z^{1/2} \Gamma^{\ast} & Z^{1/2}D_{\Gamma}
\end{pmatrix} 
\begin{pmatrix}
X^{1/2} & \Gamma Z^{1/2}\\
0 & D_{\Gamma}Z^{1/2}
\end{pmatrix}
\end{equation}
Where $X$ and $Y$ are positive matrices.
Note that this is  a Cholesky factorization-type result. This type of decomposition is used as  squared-root of matrices in the construction of the dilation. A corollary is that the operator 
$\begin{pmatrix}
I & T\\
T^{\ast} & I
\end{pmatrix}
$ is positive if and only if T is a contraction.

\begin{thm}\label{theo4}
Let X and Y be operators in $z$. The following statements are equivalent : 
\begin{itemize}
\item  There exists a contraction $\Gamma$ in $z$ such that $X = \Gamma Y$
\item
$X^{\ast}X \leqslant Y^{\ast} Y$.
\end{itemize}
\end{thm}
\begin{proof}
This result can be proved by taking the contraction $\Gamma$ with respect to $\Gamma Xh = Yh $, \cite{tseng_contractions_2006}.
\end{proof} 
As a corollary, If $X^{\ast}X = Y^{\ast}Y$ then there exists a partial isometry $V$ such that $VX=Y$. It is easy to see that we can chose $V$ to be the contraction $\Gamma$ defined above. Isometry $V$ can also be assumed unitary. For a positive operator $A\in \mathcal{L}(\mathcal{H})$, if we denote by $A^{1/2}$ its unique positive square root, then every $L$ such that $L^{\ast}L = A$ is related to $A^{1/2}$ by $A^{1/2} = VL$ (or $A^{1/2} = L^{\ast}V^{\ast}$).\\
Let us state another theorem that intervene much in Constantinescu's factorization of positive definite kernel. Note that in the next, $\mathcal{R}(\Gamma)$ will denote the close range of the operator $\Gamma$. We first start with a basic case:

\begin{thm}[row contraction]\label{identification operator square}
Let $T=[T_{1}\quad T_{2}]\ \in \mathcal{L}(\mathcal{H}_{1}\oplus \mathcal{H}_{2}, \mathcal{H} )$, then $\norm{T}\leqslant 0 $ if and only if there exists contractions $\Gamma_{1} \in  \mathcal{L} (\mathcal{H}_{1},\mathcal{H} )$ and $\Gamma_{2} \in  \mathcal{L} (\mathcal{H}_{2},\mathcal{H} )$ such that \[ T = [ \Gamma_{1} \quad D_{\Gamma_{1}^{\ast}}\Gamma_{2} ] \]
\end{thm}

\begin{proof}
The proof is a simple application of theorem \ref{theo4}. For the if part, it is obvious that we can take $\Gamma_{1}$ to be $T_{1}$. Then $\norm{T} \leqslant 1$ implies
\[  I - TT^{\ast} = I - \Gamma_{1}\Gamma_{1}^{\ast} - T_{2}T_{2}^{\ast} \geqslant 0 \] with $D^{2}_{\Gamma_{1}^{\ast}} \geqslant T_{2}T_{2}^{\ast} $. There exists so $\Delta$ such that $\Delta D_{\Gamma_{1}^{\ast}} = T_{2}^{\ast}$. Choosing $\Gamma_{2} = \Delta^{\ast}$ finishes the argument.
\end{proof}

In the same way as that of the Cholesky factorization, we can write down the defect operator for the whole contraction $T=[T_{1} \quad T_{2}]$ \cite{tseng_contractions_2006} to be
\begin{equation}
D_{T}^{2} = \begin{pmatrix}
D_{\Gamma_{1}} & 0\\
-\Gamma_{2}^{\ast} \Gamma_{1} & D_{\Gamma_{1}}
\end{pmatrix} 
\begin{pmatrix}
D_{\Gamma_{1}} & -\Gamma_{1}^{\ast}\Gamma_{2}\\
0 & D_{\Gamma_{1}}.
\end{pmatrix}
\end{equation}

Therefore, with the theorem \ref{identification operator square} we have an operator $\alpha$ such that \[ D_{T} = \begin{pmatrix}
D_{\Gamma_{1}} & 0\\
-\Gamma_{2}^{\ast} \Gamma_{1} & D_{\Gamma_{1}}
\end{pmatrix} \alpha \]
 Similarly, \[ D_{T^{\ast}}^{2} = ( D_{\Gamma_{1}^{\ast}}D_{\Gamma_{2}^{\ast}}D_{\Gamma_{2}^{\ast}}D_{\Gamma_{1}^{\ast}}  ) \]

and the general case is

\begin{thm}[Structure of  row contration]
The following are equivalent :
\begin{itemize}
\item The operator $T^{n} = [  T_{1}\quad T_{2}\quad \cdots T^{n} ]$ in $\mathcal{L}( \oplus^{n}_{k=1} \mathcal{H}_{k},\mathcal{H}' )$ is a contraction
\item
$T_{1} = \Gamma_{1}$ is a contraction and, for $k > 2$, there exists uniquely determined contractions $\Gamma_{k}\ \in \mathcal{L}( \mathcal{H}_{k},\mathcal{R}(\gamma_{k}) )$ such that $T_{k} = D_{\Gamma_{1}^{\ast}} D_{\Gamma_{2}^{\ast}} \cdots D_{\Gamma_{k-1}^{\ast}} \Gamma_{k}$
\end{itemize}
\end{thm}

Furthermore, the defect operators of the whole contraction T are of the the form
{\small 
\[ D_{T}^{2} = 
\begin{pmatrix}
D_{\Gamma_{1}} & 0 & \cdots & 0\\
-\Gamma_{2}^{\ast}\Gamma_{1} & D_{\Gamma_{2}} & \cdots & 0\\
\vdots & \vdots & \ddots & \vdots \\
-\Gamma_{n}^{\ast}D_{\Gamma_{n-1}^{\ast}} \cdots D_{\Gamma_{2}^{\ast}} & -\Gamma_{n}^{\ast}D_{\Gamma_{n-1}^{\ast}} \cdots D_{\Gamma_{3}^{\ast}}\Gamma_{2} & \cdots & D_{\Gamma_{n}}
\end{pmatrix}
\begin{pmatrix}
D_{\Gamma_{1}} & -\Gamma_{1}^{\ast}\Gamma_{2} & \cdots & -\Gamma_{1}^{\ast}D_{\Gamma_{2}^{\ast}} \cdots D_{\Gamma_{n-1}^{\ast}}\Gamma_{n}\\
0 & D_{\Gamma_{2}}  & \cdots & -\Gamma_{2}^{\ast}D_{\Gamma_{3}^{\ast}} \cdots D_{\Gamma_{n-1}^{\ast}}\Gamma_{n}\\
\vdots & \vdots & \ddots & \vdots \\
0 & 0 & \cdots & D_{\Gamma_{n}}
\end{pmatrix}
\]
}
and \[ D_{T^{\ast}}^{2}  = D_{\Gamma_{1}^{\ast}} \cdots D_{\Gamma_{n}^{\ast}}D_{\Gamma_{n}^{\ast}} \cdots D_{\Gamma_{1}^{\ast}} \]
\begin{proof}
It can be proved straightforwardly by induction.
\end{proof}
This construction permits to understand the apparition of the operators $\alpha$ and $\beta$ in the publications of Constantinescu which are used to identify the defect space of the components (the underlying contractions of a row contraction) of a row contraction with the defect space of the row contraction itself.  Same results are readily obtained for column contraction of the form $T = \begin{pmatrix}
T_{1}\\ \vdots \\ T_{2}
\end{pmatrix} $.

%=========================================================================================================

\section{Geodesic equation in the space of curve $ \mathcal{M} $}\label{AnnexB}

To have a complete insight on the geodesic equation, we give the proof for a more general case that arise when considering the SRV and not only the TSRV function of a cuve, that is the curve are parametrized by their starting point and their velocity but their starting point are not transported to the identity. 

\begin{thm}
A path of curves  $ [0,1] \ni s \mapsto \left( c(s,0) , q(s,t) \right)$ (t is the parameter of the curve $ c(s,\cdot) $)is a geodesic on $ \mathcal{M} $ if and only if:
\begin{align}
\nabla_{\partial c / \partial s  }c(s,0) + \int_{0}^{1}{ \mathcal{R} \left(  q(s,t), \nabla_{\partial c / \partial s  }q(s,t)  \right)(c(s,0)) dt} & = 0 \quad \forall s \\
\nabla_{\partial c / \partial s  } \left( \nabla_{\partial c / \partial s } q(s,t) \right) (s,t) & = 0 \quad \forall s,t
\end{align}
 \end{thm}

Similarly to \cite{brigant_computing_2016} and \cite{zhang_video_based_2015}, we consider a variation of the path $ s \mapsto c(s,0), q(s,t) $ starting and ending at the same points, we denote $ \left\{ \left( c(s,0,\tau) ,q(s,t,\tau) \right) \right\} $. In figure (\ref{geodesic variations}), to get a clear picture,  we have represented a variation of a path of curves with fixed starting and ending points. Though similar, the situation here is a bit different because of the representation of the curve through its SRV function, which we can hardly represent. But the process remains similar. We emphasize the subtle difference with \cite{brigant_computing_2016}. Here, we work directly in the tangent space representation, via the SRV representation, and not with "the whole family" of curves $ c(s, t, \tau) $. 
\begin{figure}
\label{path of curves}
\center
\includegraphics[width = 10cm]{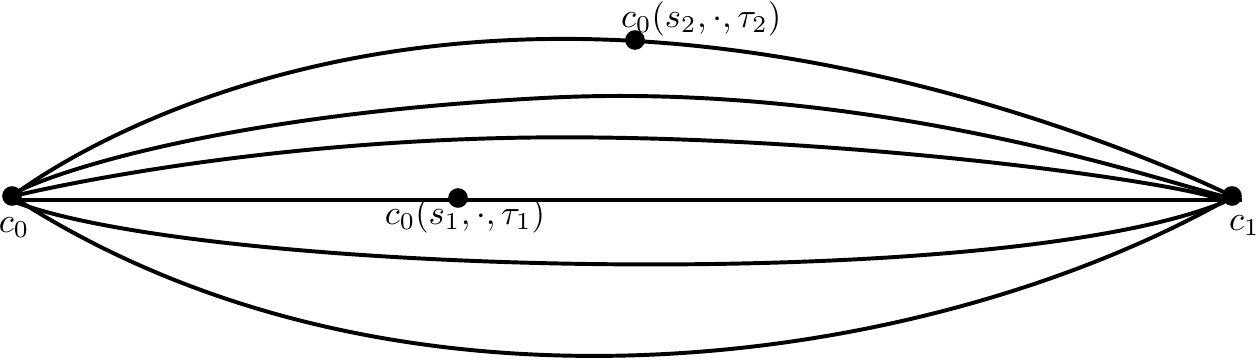}
\caption{we consider a beam of curves, which consists in a slight modification of the geodesic. The different curves are indexed by $\tau$. The idea is to find which of these curves give the minimal energy to go from $ c_{0} $ to $ c_{1} $}
\label{geodesic variations}
\end{figure}
We denote $ \partial_{\tau}c(s,0, \tau) =  \dfrac{ \partial c(s,0,\tau) }{ \partial \tau } $, and similarly for $ \partial_{s}c(s,0,\tau) $ and $ \partial_{\tau}c(s,0,\tau) $. The energy of the path indexed by $ \tau $ is 
\[ 
E(\tau) =\dfrac{1}{2} \int_{0}^{1}{ \inner{\partial_{s} c(s,0,\tau)}{ \partial_{s} c(s,0,\tau )}  + \inner{\nabla_{\partial c / \partial s}q(s,t,\tau)}{ \nabla_{\partial c / \partial s}q(s,t,\tau)}  ds}.
 \]
Recall that the derivative of the inner product is given by $ \dfrac{d}{dx}\inner{f(x)}{f(x)} = 2* \inner{f(x)}{ \dfrac{df}{dx} }$. Then
\[ 
E'(0) = \int_{0}^{1}{ \inner{ \nabla_{\partial c / \partial \tau} \dfrac{\partial c}{\partial s}(s,0,0) }{ \dfrac{\partial c}{\partial s}(s,0,0)} + \inner{ \nabla_{\partial c / \partial \tau} \nabla_{\partial c / \partial s} q(s,t,0) }{ \nabla_{\partial c / \partial s}q(s,t,0)} ds }
 \]
with $ \nabla_{\partial c / \partial s} \left(   \partial_{\tau}c(s,0, \tau) \right)  = \nabla_{\partial c / \partial \tau} \left(  \partial_{s}c (s,0,\tau) \right)$  and owing to the curvature tensor
\newline
 $ \mathcal{R}\left( \partial_{\tau}c(s,0,\tau), \partial_{s}c(s,0, \tau) \right) (q(s,t,\tau) = \nabla_{\partial c / \partial \tau}\nabla_{\partial c / \partial s}(q(s,t,\tau)) - \nabla_{\partial c / \partial s}\nabla_{\partial c / \partial \tau}(q(s,t,\tau)) $
we have
\[ 
\begin{split}
E'(0) & = \int_{0}^{1} \inner{ \nabla_{\partial_{s} c} \partial_{\tau}c(s,0,\tau) }{ \partial_{s}c(s,0,\tau) } +
 \inner{ \mathcal{R} \left( \partial_{\tau}c(s,0,\tau), \partial_{s}c(s,0,\tau) \right) q(s,t,\tau) }{ \nabla_{\partial_{s}}q(s,t,\tau) } \\
&  +\inner{ \nabla_{\partial_{s} c} \nabla_{\partial_{\tau} c } q(s,t,0) }{ \nabla_{\partial_{s} c}q(s,t,0)} ds
\end{split}.
 \]
Integrating by parts now, allows to have

\begin{align*}
\int_{0}^{1}{ \inner{ \nabla_{\partial_{\tau}c} \partial_{s}c(s,0,\tau) }{\partial_{s}c(s,0,\tau)}ds } & = - \int_{0}^{1}{ \inner{ \nabla_{\partial_{s}c} \partial_{s}c(s,0,\tau) }{\partial_{\tau}c(s,0,\tau)}ds }\\
\int_{0}^{1}{ \inner{ \nabla_{\partial_{s} c} \nabla_{\partial_{\tau}c} (q(s,t,\tau))}{\nabla_{\partial_{s}}q (s,t,\tau) } } & =  - \int_{0}^{1}{ \inner{ \nabla_{\partial_{s} c} \nabla_{\partial_{s}c} (q(s,t,\tau))}{\nabla_{\partial_{\tau}}q (s,t,\tau) } }
\end{align*}

which yields to:

 \[ 
\begin{split}
E'(0) & = \int_{0}^{1}  (- \inner{ \nabla_{\partial_{s} c} \partial_{s}c(s,0,\tau) }{ \partial_{\tau}c(s,0,\tau) }) +
 \inner{ \mathcal{R} \left( \partial_{\tau}c(s,0,\tau), \partial_{s}c(s,0,\tau) \right) q(s,t,\tau) }{ \nabla_{\partial_{s}}q(s,t,\tau) } \\
&  +( -\inner{ \nabla_{\partial_{s} c } \nabla_{\partial_{s} c } q(s,t,0) }{ \nabla_{\partial_{\tau} c }q(s,t,0)}) ds
\end{split}
 \]

For any vector fields $ X,Y,Z,W $, $ \inner{ \mathcal{R}(X,Y)Z }{W} = - \inner{ \mathcal{R}(W) }{Z} $, consequently we obtained 

\[ 
\begin{split}
E'(0) & = - \int_{0}^{1} \inner{ \nabla_{\partial_{s} c} \partial_{\tau}c(s,0,\tau) }{ \partial_{s}c(s,0,\tau) } +
 \inner{ \mathcal{R} \left( q(s,t,\tau), \nabla_{\partial_{s}}q(s,t,\tau) \right) ( \partial_{s}c(s,0,\tau) ) }{ \partial_{\tau}c(s,0,\tau } \\
&  +\inner{ \nabla_{\partial c / \partial s} \nabla_{\partial c / \partial \tau} q(s,t,0) }{ \nabla_{\partial c / \partial s}q(s,t,0)} ds
\end{split}
 \]
Geodesic corresponds to minimal energy. It means that every other path that starts and ends at the same points should require more energy to travel than the geodesic. We then have to solve $ E'(0) = 0 $ for every $ \partial_{\tau}c(s,0,\tau) $ and every $ \nabla_{\partial_{\tau}}(q(s,t,\tau)) $. This gives the result. \\
Now when the framework is given by the TSRV and not by the SRV, only the second part of the geodesic equation remains due to the fixed starting point which corresponds to the identity element. This simplifies a lot the equation, even though the derivation is the same.

\bibliography{Dilation2}
\bibliographystyle{plain}

\end{document}